\documentclass[12pt]{amsart}
\usepackage{amsmath,amssymb,amsbsy,amsfonts,amsthm,latexsym,mathabx,
            amsopn,amstext,amsxtra,euscript,amscd,stmaryrd,mathrsfs,
            cite,array,mathtools,enumerate}

\usepackage{url}
\usepackage[colorlinks,linkcolor=blue,anchorcolor=blue,citecolor=blue,backref=page]{hyperref}

\usepackage{color}
\usepackage{a4wide}

\usepackage{color}
\usepackage{float}

\hypersetup{breaklinks=true}

\usepackage[english]{babel}

\usepackage{mathtools}
\usepackage{todonotes}

\usepackage[norefs,nocites]{refcheck}

\usepackage{enumitem}

\allowdisplaybreaks

\usepackage{todonotes}
\usepackage{url}
\usepackage[colorlinks,linkcolor=blue,anchorcolor=blue,citecolor=blue,backref=page]{hyperref}
\usepackage[noabbrev,capitalize]{cleveref}

\usepackage{tikz}
\usetikzlibrary{arrows.meta}
\usepackage{tikz-cd}

\newtheorem{theorem}{Theorem}
\newtheorem{lemma}[theorem]{Lemma}

\newtheorem{prop}[theorem]{Proposition}

\numberwithin{equation}{section}
\numberwithin{theorem}{section}

\newcommand{\Z}{\mathbb{Z}}

\newcommand{\U}{U}

\newcommand{\X}{Z}

\newcommand{\Fq}{\mathbb{F}_q}
\newcommand{\Fqn}{\mathbb{F}_{q^n}}
\newcommand{\Fqp}{\mathbb{F}_{q^n}}
\newcommand{\Fqm}{\mathbb{F}_{q^m}}
\newcommand{\Fqb}{\overline{\mathbb{F}}_q}

\newcommand{\Bf}[1]{\boldsymbol{#1}}
\newcommand{\mb}{\Bf {m}}
\newcommand{\f}{f}
\newcommand{\F}{F}
\newcommand{\qf}{\mathfrak{q}}

\def\cO{{\mathcal O}}

\def\cR{{\mathcal R}}
\def\cS{{\mathcal S}}

\title[HEC Koblitz]{Linear complexity of sequences on Koblitz curves of genus 2}

\author[V.~Anupindi]{Vishnupriya Anupindi}
\address{Johann Radon Institute for Computational and Applied Mathematics, Austrian Academy of Sciences,  Altenberger Stra\ss e 69, A-4040 Linz, Austria} 
\email{vishnupriya.anupindi@oeaw.ac.at}

\begin{document}

\maketitle

\begin{abstract}
	In this paper, we consider the hyperelliptic analogue of the Frobenius endomorphism generator and show that it produces sequences with large linear complexity on the Jacobian of genus $ 2 $ curves.
\end{abstract}

\section{Introduction}
An important operation in elliptic curve based cryptosystems is to compute scalar multiples of a given group element. The standard method for computing scalar multiples is the \textit{double-and-add-method}, but faster methods have been suggested by using the Frobenius endomorphism on special curves known as Koblitz curves, see \cite{Kob_CM_MR1243654, Muller_MR1648949, Smart_MR1685175, Solinas_MR1759617}. The ideas for fast computation of scalar multiples on elliptic Koblitz curves have been generalized to hyperelliptic curves of genus $2$, see \cite{Lange_Stein_MR1895585}. 

In \cite{ShparlinskiLange05}, Lange and Shparlinski investigated the problem of choosing random elements from elliptic and hyperelliptic curves, see also \cite{lange2007distribution, LM_EC_FEG_MR3649088}. One can choose such elements by computing random scalar multiples of an initial element fixed in advance. However, Lange and Shparlinski \cite{ShparlinskiLange05}, by taking advantage of fast computation of scalar multiplication on Koblitz curves, introduced a more efficient and direct way to obtain random-looking elements, called \emph{Frobenius endomorphism generator}. 

In this paper, we study some properties of pseudorandomness of sequences derived from hyperelliptic curves of genus $ 2 $ using the Frobenius endomorphism generator. In particular, we investigate the level of randomness of such sequences in terms of linear complexity.
We recall, that the \emph{linear complexity} of a sequence $(s_n)$ of length $N$ over the finite field $\Fq$ is defined as the smallest non-negative integer $L$ such that the first $ N $ terms of the sequence $ (s_n) $ can be generated by a linear recurrence relation over $ \Fq $ of order $ L $, i.e.\ there exist $ c_0, c_1, \dots , c_{L-1} \in \Fq $ such that 
\begin{equation*}
s_{n+L} = c_0s_{n} + c_1s_{n+1} + \dots  + c_{L-1}s_{n+L-1},\quad  0 \leq n \leq N-L-1.
\end{equation*}
The linear complexity measures the unpredictability of a sequence, hence for applications in cryptography, a large linear complexity is desired. However, a large linear complexity is not a sufficient condition for the unpredictability of a sequence. For more details, see \cite{MeidlWinterhof,Nied1, Winterhof2010}. 

In \cref{sec:HEC}, we recall some properties of hyperelliptic curves and in \cref{sec:Kob_Curve}, we define the Frobenius endomorphism generator and state the main result. In \cref{sec:prep}, we collect auxiliary results which are used in the proof. In particular, we recall the Grant representation \cite{grant1990formal} of the Jacobian of a hyperelliptic curve of genus $ 2 $ and some results from \cite{anupindiLinearComplexitySequences2021}. Finally, in \cref{sec:proof}, we prove the main result. 

\section{Hyperelliptic curves}\label{sec:HEC}
Let $ \Fq $ be a finite field with characteristic $ p \geq 3 $ and $ \Fqp $ be an extension field of $ \Fq $ with $ n \geq 1 $. Let $\Fqb$ be the algebraic closure of $\Fq$. 
\subsection{Points on hyperelliptic curves}
Let $ C $ be a hyperelliptic curve of genus $ g \geq 1 $ defined over the base field $ \Fq $ by
\begin{equation}\label{def:hec}
C: Y^2 = h(X),
\end{equation}
where $ h(X) \in \Fq[X]$ is a polynomial of degree $ 2g+1 $.  For details on hyperelliptic curves, see \cite{cohen2005handbook,Galbraith-book, Koblitz-book}.
We denote the $\Fqp$-rational points of $C$ by $C(\Fqp)$, which are the solutions over $\Fqp$ of the defining equation \eqref{def:hec} together with a point $\cO$ at infinity. 
By the Hasse-Weil bound \cite[Theorem~5.2.3]{stichtenothAlgebraicFunctionFields2009}, we have
\begin{equation}\label{eq:number_of_points}
\, |\, |C(\Fqp)| -(q^n+1) |\leq 2g q^{n/2}.
\end{equation} 
%The additive group associated to curve $ C $ that we use in cryptography is the \emph{Jacobian} $J_C$ of the curve $C$, which is an abelian variety of dimension $ g $ \cite[Theorem~A.8.1.1]{hindryDiophantineGeometryIntroduction2000}. For elliptic curves (curves with genus $ g=1 $), the Jacobian $ J_C $ is isomorphic to the curve $ C $, but this is not the case for genus $ g \geq 2 $.  For higher genus hyperelliptic curves $ (g \geq 2) $, we can describe elements of the Jacobian $ J_C $ in the following way. 
\subsection{Jacobian of hyperelliptic curves}
For an affine point $P=(x,y)\in C$, we write $-P=(x,-y)$ and $-\cO=\cO$ for the point at infinity. A \emph{divisor} $D$ of $C$ is an element of the free abelian group over the points of $C$, e.\ g.\ $D =\sum_{P\in C } n_P P$  with $n_P\in \mathbb{Z}$ and $n_P = 0$ for almost all points $P$. A \emph{reduced divisor} is given by 
\begin{equation}\label{eq:reduced_div}
D = P_1+\dots +P_r-r\cO,
\end{equation}
where $1\leq r \leq g$, $P_1, \dots, P_r\in C$, $P_i\neq \cO $ for $ 1\leq i \leq r$ and $P_i\neq -P_j$ for $1\leq i<j\leq r$. 

The \emph{Jacobian} $J_C$ of the curve $C$ is the set of reduced divisors. One can define an addition operation on the set of reduced divisors, denoted by $+$, with the identity element $\cO$, which makes $J_C$ into a group.
The elements of the curve $C(\Fq)$ are represented in the Jacobian by the set
\begin{equation}\label{eq:Theta}
\Theta(\Fq)=\{D\in J_C(\Fq): D=P-\cO, P\in C(\Fq) \} \cup \{\cO\} .
\end{equation}
We also write $\Theta=\Theta(\Fqb)$.

The Frobenius endomorphism $ \sigma: \Fqb \rightarrow \Fqb, x \mapsto x^q $, extends naturally to points on $ C $, where $ \sigma((x,y)) = (x^q,y^q)  $ and $ \sigma(\cO) = \cO $. For $ D = \sum_{i=1}^{r} P_i - r \cO  \in J_C $, define $ \sigma(D) = \sum_{i=1}^{r} \sigma(P_i) - r \cO $. An element $D \in J_C$ as given in \eqref{eq:reduced_div} is said to be defined over $\Fq$ if $ \sigma(D) $ permutes the set $\{P_1,\dots, P_r\}$. We use $J_C(\Fq)$ to denote the set of elements of $J_C$ which are defined over $\Fq$. 

The characteristic polynomial, $ \chi_C (T) $ of the Frobenius endomorphism $ \sigma $ is a degree $ 2g $ polynomial with integer coefficients of the following form
\begin{equation}\label{eq:char_poly_Frob}
\chi_C (T) = T^{2g} + s_1T^{2g-1} + \dots + s_gT^g + \dots + s_1q^{g-1}T + q^g, s_i \in \Z.
\end{equation}
It follows from the Hasse-Weil Theorem \cite[Theorem 5.1.15 and 5.2.1]{stichtenothAlgebraicFunctionFields2009}, that the complex roots $ \tau_i $ of $ \chi_C $ have absolute value $ |\tau_i| = q^{1/2}, i = 1, \dots, 2g $.  For any extension degree $ n $, the cardinality of $ J_C(\Fqn) $  is given by 
\begin{equation}\label{eq:card_Jac_roots}
|J_{C}(\Fqn)| = \prod_{i=1}^{2g} (1-\tau_i ^n).
\end{equation}
In particular, we have
\begin{equation}\label{eq:size_of_jac}
(q^{n/2}-1)^{2g}\leq |J_{C}(\Fqn)|\leq (q^{n/2}+1)^{2g}, ~  n\geq 1.
\end{equation}
\subsection{Mumford representation}
A compact representation of elements of the Jacobian $ J_C $ is given by the \emph{Mumford representation} \cite{mumfordTataLecturesTheta2007a} using a pair of polynomials $ [u,v] \in \Fq[X] \times \Fq[X]$. For a reduced divisor $ D = \sum_{i=1}^{r} P_i - r \cO $ with $ P_i = (x_i,y_i) $ the Mumford representation is given by $ u =\prod_{i=1}^r(X-x_i)$ and $v$ such that $v$ interpolates the points $ P_i $ respecting multiplicities. In particular,
\renewcommand{\theenumi}{\alph{enumi}}
\begin{enumerate}
	\item \label{mum-a} $u$ is monic,
	\item $u$ divides $f-v^2$,
	\item $\deg(v)< \deg(u)\leq g$.
\end{enumerate}
%$ u $ is monic, $ v $ is either $ 0 $ or $ \deg(v) < \deg(u) \leq g$ and $ u $ divides $ f - v^2 $. 
For genus $ g=2 $, a generic element $ D = P_1 + P_2 - 2\cO $ is represented by the polynomials 
\begin{equation}\label{eq:Mumford_rep}
u = x^2 + u_1x + u_0,~ v= v_1x + v_0,~ u_i, v_i \in \Fqn, i \in \{0,1\} \text{ such that } D=[u,v].
\end{equation} 

\section{Koblitz curves and fast generation of elements in Jacobian} \label{sec:Kob_Curve}
By a hyperelliptic \emph{Koblitz curve}, we refer to a hyperelliptic curve that is defined over a small finite field and is considered over a large extension field. In this work, we avoid fields with characteristic $ 2 $ for technical reasons. For Koblitz curves, it is recommended to choose base fields $ q \leq 7 $ for computational advantage, see \cite{lange2005koblitz}. However, we do not impose this restriction for our result. 
%In this work, we choose $ 3 \leq q \leq 7 $. Typically, we choose $ \Fq, \old{q \leq 7} $ and $ \Fqn $ with $ n $ large, \warn{i.e. $ n \sim 10^m, m\geq 2 $.} For more details, see \cite{lange2005koblitz}. 

For fast generation of elements in the Jacobian $ J_C(\Fqn) $, Lange and Shparlinski in \cite{ShparlinskiLange05} introduced the following method using the Frobenius endomorphism. Here we restrict ourselves to the genus $ 2 $ case.
Let 
\begin{equation*}
\cR = \{ 0, \pm 1, \dots, \pm (q^2 -1)/2 \}
\end{equation*}
represent the set $\Z /q^2\Z$. Let $D \in J_C(\Fqn)$ be an element of order $\ell$. For fixed $k \leq n$, consider the element of $J_C$ defined as follows:
\begin{equation}\label{eq:Dm_def}
D_{\mb} = \sum_{j=0}^{k-1} m_j \sigma^{j}(D), \quad \mb = (m_0, \dots,m_{k-1}) \in \cR^k. 
\end{equation}
%In the above construction, collisions can occur, that is, different multi-indices can yield the same element. However in \cite{ShparlinskiLange05}, the authors started to investigate the randomness properties of $ D_{\mb} $. In particular, they showed that $ D_{\mb} $ is well distributed, namely, there are few collisions. For elliptic curves, the randomness properties of the points $  D_{\mb} $ were investigated in \cite{lange2007distribution} and \cite{LM_EC_FEG_MR3649088}. \warn{More content to be added...}
It is natural to expect that the divisors $ D_{\mb} $ defined by \eqref{eq:Dm_def} are sufficiently uniformly distributed. Lange and Shparlinski \cite{ShparlinskiLange05} showed that $ D_{\mb} $ do not take the same value too often (which would otherwise have catastrophic implications for their cryptographic applications).

In this paper we further investigate the randomness properties of $ D_{\mb} $. Namely, we show that different statistics of the divisors $ D_{\mb} $, like the Mumford coordinates $ u_i, v_i $ as in \eqref{eq:Mumford_rep}, possess large linear complexity if the divisors $ D_{\mb} $ are arranged in a natural way, say in lexicographic ordering. More precisely, let 
$\f \in \Fqp(J_C)$ be a rational function in the function field of the Jacobian.  We arrange the elements of $\cR^k$ with a lexicographic ordering and define the sequence $ (w_{\mb})_{\mb \in \cR^k} $ with
\begin{equation}\label{eq:seq_def}
w_{\mb} = \begin{cases}
f(D_{\mb}) &\text{if $D_{\mb}$ is not a pole of $f$,} \\
0  &\text{otherwise.}
\end{cases} 
\end{equation}
%We also define $ \mb + i $ as the $ i^{\text{th}} $ element after $ \mb $ with respect to the lexicographic ordering. 
Throughout the paper, $U \ll V $ is equivalent to the inequality $ |U| \leq cV $ with some constant $ c > 0 $.
Our main result is the following bound on the linear complexity of $ (w_{\mb})_{\mb \in \cR^k} $.
\begin{theorem}\label{thm:main_result}
	Let $ C $ be a hyperelliptic curve of genus $ 2 $, defined over the base field $ \Fq $ and let $J_C(\Fqn)$ be its Jacobian over the extension field $ \Fqn $. Let the characteristic polynomial of the Frobenius endomorphism $\chi_C$ be irreducible. Let $f \in \Fqn(J_C)$ be a rational function with pole divisor of the form $\alpha \Theta, \alpha \in \Z, \alpha \geq 1 $. If $D \in J_C(\Fqn)$ is of prime order $\ell, \ell \nmid q^2$, then for any $ k $ where $ 1 \leq k \leq n $ with $(w_{\mb})_{\mb \in \cR^k}$ as defined in \eqref{eq:seq_def}, we have
	\begin{equation}\label{eq:thm_main}
	L(w_{\mb}) \gg  \frac{\min\{ q^{3k/2},\ell/q^8 \}}{ q^n \deg \f }.
	\end{equation} 
\end{theorem}
The result is non-trivial if $ k \geq 2n/3 $ and $ \ell \geq q^{n+8} $. In the ideal case, $k=n$, $ \deg \f =1 $ and $ \ell \sim q^{2n} $, we obtain $ L(w_{\mb}) >  c q^{n/2} $ for some constant which may depend on $ \deg \f $. Examples for rational functions with $ \deg \f =1 $ are the Mumford coordinates \eqref{eq:Mumford_rep}.

We assume the characteristic polynomial of Frobenius endomorphism $ \chi_C$ to be irreducible, in particular, $ \chi_C $ is irreducible over $ \Z $. Practically, this is the most interesting case, since, by \eqref{eq:card_Jac_roots} any non-trivial factor of $ \chi_C $ leads to a non-trivial factor of the group order, which we want to avoid. 

We remark, that in \eqref{eq:Dm_def}, if we replace the Frobenius map $ \sigma $ with the multiplication map $ [2]: D \mapsto 2D $, and if we use colexicographic ordering for arranging sequence elements $ D_{\mb} $, then we are in the linear congruential generator case, for which we proved a stronger bound in \cite{anupindiLinearComplexitySequences2021}.

We also remark, that Lange and Shparlinski \cite{ShparlinskiLange05, lange2007distribution} defined and investigated the randomness properties of similar, but not completely analogous point-set for the elliptic curve case. Later, M{\'e}rai \cite{LM_EC_FEG_MR3649088} studied the randomness properties of a sequence of elements from this point set,  by arranging elements in a sequence using lexicographic ordering.

The proof of \cref{thm:main_result} is based on the method of \cite{LM_EC_FEG_MR3649088}, the results of \cite{ShparlinskiLange05} and taking advantage of the explicit addition formulas for genus $ 2 $ provided by Grant \cite{grant1990formal}. 

%This is indeed the case when $ \f $ is one of the Mumford coordinates, where $ \deg \f =1$.
% for $ \f = u_0 = -z_{12}, u_1 = -z_{22}, v_0 = z_{122}, v_1 = z_{222} $. \commV{Not perfect %yet, $z_{ijk}$ not defined at this point.}

\section{Preparation}\label{sec:prep}
The aim of this section is to collect some technical results for the proof of the main theorem. We use the \emph{Grant representation} of a hyperelliptic curve of genus $ 2 $ since it provides explicit addition formulas. This allows us to prove the degree estimate in \cref{prop:F_nonconst}. 
%In this section, we will recall the \emph{Grant representation} \new{of a hyperelliptic curve of genus $ 2 $} and some other technical lemmas from \cite{anupindiLinearComplexitySequences2021}. The aim of this section is to prove \cref{prop:F_nonconst}, which is a key ingredient for proof of the main theorem.

\subsection{Arithmetic for genus 2 using Grant representation}
In order to implement the group law $(Q,R) \mapsto Q+R$ on the Jacobian, one can use Cantor's algorithm which uses the Mumford representation. However, this algorithm is implicit.
%Algorithms based on Cantor's algorithm use the Mumford representation to implement the group law $(Q,R) \mapsto Q+R$ on the Jacobian. 
In this work, we use the explicit addition formulas provided by Grant \cite[Theorem 3.3]{grant1990formal}.

Let $C$ be the hyperelliptic curve of genus $g = 2$ defined by \eqref{def:hec} with
$$
h(X) = X^5 + b_1X^4 + b_2X^3 + b_3X^2 + b_4X + b_5 \in \Fq[X],
$$
for the finite field $\Fq$ with characteristic $p \geq 3 $.
In \cite{grant1990formal}, Grant provides an embedding of $J_C$ into the projective space $\mathbb{P}^8$.

Let 
\begin{equation} \label{eq:poly_ring}
\Fq[\mathbf{\X}] = \Fq[\X_{11},\X_{12},\X_{22},\X_{111},\X_{112},\X_{122},\X_{222},\X]
\end{equation}
be a polynomial ring over $\Fq$ in $8$ variables. The following proposition gives us a set of defining equations for the Jacobian, see \cite[Corollary 2.15]{grant1990formal}.

\begin{prop}
	There are polynomials $ f_1, \dots , f_{13} \in \Fq[\mathbf{\X}] $ 
	such that
	\begin{equation*}
	J_C \cong V(f_1^h,\dots,f_{13}^h) = \{ z \in \mathbb{P}^8 : f_i^h(z) = 0, 1 \leq i \leq 13 \},
	\end{equation*}
	where $f_i^h$ denotes the homogenized polynomial with respect to the variable $\X_0$.
	Moreover, an embedding $\iota : J_C \rightarrow \mathbb{P}^8$ is given by
	\begin{equation}\label{eq:iota}
	\iota(D) = 
	\begin{cases}
	(1:z_{11}:z_{12}:z_{22}:z_{111}:z_{112}:z_{122}:z_{222}:z)  & \text{if~} D \in J_C\setminus \Theta, \\ 
	(0:0:0:0:1:0:0:0:0) & \text{if~} D = \cO ,\\
	(0:0:0:0:-x^3: - x^2 : -x : 1 : -y) & \text{if~} D = P - \cO \in \Theta \setminus\cO,
	
	\end{cases}
	\end{equation}
	where $ P = (x,y) $.
\end{prop}
See \cref{sec:def_eq_J} for the polynomial expressions of $f_1,\dots, f_{13}$ in the same notation as used in this work. For $D = (x_1,y_1) + (x_2,y_2) - 2\cO \in J_C(\Fq)\setminus \Theta(\Fq)$, the components $z_{jk}, z_{jkl}$ of $\iota(D)$ can be expressed as rational functions in the coordinates $(x_1,y_1)$ and $(x_2,y_2)$. 

We denote the affine part of $J_C$ with respect to variable $\X_0$ under $\iota$ by $U$. Then \\ $ U = J_C\setminus \Theta.$
Moreover, 
by \cite[Theorem~2.5]{grant1990formal}, we have
\begin{equation}\label{eq:def_U}
\U \cong V(f_1,\dots, f_6)
\end{equation}
Since $ J_C$ is irreducible and has dimension $ 2 $, it follows that $ U $
is irreducible, dense, and has dimension $ 2 $, see \cite[Example 1.1.3 ]{hartshorneAlgebraicGeometry1977}. 
As a result, $ \Fq(U) = \Fq(J_C)$, see \cite[Theorem~3.4]{hartshorneAlgebraicGeometry1977}.

For a rational function $ h \in \Fq(U) $, we define its degree by choosing a representative element $\frac{h_1}{h_2}$ of the equivalence class $ h $, such that $\deg h_1$ is minimal and set 
$$
\deg h =\max\{\deg h_1, \deg h_2\}.
$$

We summarize the algebraic properties of the group law in the Grant representation. For explicit expressions, see \cref{sec:addition_formulas}.
\begin{lemma}\label{lem:Grant_q}
	Assume that $Q,R,Q+R,Q-R \in U$. Let 
	\begin{equation}\label{eq:q}
	\qf(Q,R) = z_{11}(Q)-z_{11}(R)+z_{12}(Q)z_{22}(R)-z_{12}(R)z_{22}(Q).
	\end{equation}
	Then there are explicit formulas for $ z_{jk}(Q+R), z_{jkl}(Q+R) $ which are rational functions in $ z_{jk}(Q),z_{jk}(R), z_{jkl}(Q) z_{jkl}(R) $ and $ \qf(Q,R) $ for $ 1 \leq j \leq k \leq l \leq 2 $.
\end{lemma}

We recall \cite[Lemma 2.3]{anupindiLinearComplexitySequences2021} which will be used in \cref{prop:F_nonconst}.

\begin{lemma}\label{lem:Grant_q_dim}
	Assume that $Q,R,Q+R,Q-R \in U$. Let $ \qf(Q,R) $ be defined by \eqref{eq:q} and set $\qf_R(Q)=\qf(Q,R) $.
	Then for any fixed $R\in \U$,  the zero set $\{ \qf_R(Q)=\qf(Q,R)=0 \}$ has dimension one and $\Theta \pm R\subset\{\qf_R=0\}$. Moreover
	if $R' \in U$ with $R\neq \pm R'$, then 
	\begin{equation}\label{eq:cardinality_zeros_q_1}
	|\{\qf_{R} = 0\} \cap \{\qf_{R'} = 0\} \cap \U| \leq 20.
	\end{equation}
\end{lemma}

One can show that for $D \in U(\Fqm)$, where $ m \in \Z, m \geq1 $, we have $ |\{\Theta(\Fqm) + D\} \cap \Theta(\Fqm)| \leq 2 $. See \cite[Lemma~2.4]{anupindiLinearComplexitySequences2021}. Thus
\begin{equation}\label{eq:Theta_intersection}
|\{\Theta(\Fqm) + D\} \, \cap \,U| \geq |\Theta(\Fqm)|-2.
\end{equation}

\begin{prop}\label{prop:F_nonconst}
	Let $\f \in \Fqp(U)$ be a rational function with a pole divisor of the form $\alpha \Theta, \alpha \in \Z, \alpha \geq 1 $. Let $L$ be a positive integer and 
	%\begin{equation}\label{eq:bound_L}
	%L < \frac{|\Theta(\Fqp)| - 2}{20}. 
	%\end{equation}
	let $R_0, \dots, R_L \in J_C(\Fqp)$ such that $R_i \notin \Theta(\Fqp)$ and $R_L \neq \pm R_j$ for $ 0 \leq i \leq L $ and $ 0 \leq j \leq L-1$. Let $c_0. \dots, c_L \in \Fqp$ with $c_L \neq 0$. Then the rational function $\F \in \Fqp(U)$, with 
	\begin{equation*}
	\F(Q) = \sum_{i=0}^{L} c_i \f(Q + R_i )
	\end{equation*}
	is non-constant and has degree 
	\begin{equation}\label{eq:deg_F}
	\deg \F \leq 6(L+1) \deg \f.
	\end{equation}
\end{prop}
\begin{proof}
	Defining the function $\f_{R_i} : Q \mapsto \f(Q + R_i )$ yields 
	$$
	\F(Q) = \sum_{i=0}^{L-1} c_i f_{R_i}(Q) + c_L f_{R_L}(Q).
	$$ 
	To prove that $ \F $ is non-constant, we show that there exists $ Q \in U $ such that it is a pole of $ \f_{R_L} $, but not a pole of any other terms $ \f_{R_i}$ for $ i < L $.
	
	Observe that $\f_{R_L} $ has a pole at $ Q $ when $ Q \in \Theta - R_L $, in particular, when $ Q \in \Theta(\Fqm) - R_L $, for $ m \geq 1 $ independent of $ n $. Define $\qf_{R_i} = \qf(Q,R_i)$. 
	From \cref{lem:Grant_q_dim}, we know that $ \Theta(\Fqm) -R_i \subseteq \{\qf_{R_i} = 0\} $.  Hence, by \eqref{eq:cardinality_zeros_q_1} we obtain
	$$
	\left| \Big((\Theta(\Fqm) - R_L) \cap \U \Big) \cap \{\qf_{R_i} = 0\} \right|\leq
	\left| \{\qf_{R_L}=0\} \cap \{\qf_{R_i} = 0\} \cap \U\right|\leq 20 .
	$$
	Thus, by \eqref{eq:Theta_intersection}, we obtain
	\begin{align}\label{eq:set}
	&\left|\Big((\Theta(\Fqm) - R_L)\cap \U \Big) \setminus \left( \bigcup\limits_{i=0}^{L-1} \{\qf_{R_i} = 0 \} \right)\right| \\ \notag
	=&\left| \bigcap\limits_{i=0}^{L-1} \bigg( \Big((\Theta(\Fqm) - R_L) \cap \U \Big)  \setminus \{\qf_{R_i} = 0\} \bigg)\right| \geq |\Theta(\Fqm)|-2-20L.
	\end{align}
	We pick $ m $ such that $ |\Theta(\Fqm)|-2-20L~ > 0 $. Hence, there exists a point $Q$ which is a pole of $ \f_{R_L}$ but not a pole of any other term of $\F$. Hence, $\F$ is non-constant.
	
	To estimate the degree of $\F$, we first estimate the degree of $\f_{R_i}$. For arbitrary $i$, define $R = R_i$. 
	We define 
	\begin{equation*}
	z_{jk}^R(Q) = z_{jk}(Q+R),
	\quad
	z_{jkl}^R(Q) = z_{jkl}(Q+R),
	\quad
	z^R = z(Q+R).
	\end{equation*}
	Then we can write $ \f_{R}(Q) $ as 
	\begin{equation*}
	\f_{R}(Q) =\f \big(Q + R\big) =\f \big(z_{11}^{R}(Q), \dots , z_{222}^{R}(Q),z^{R}(Q)\big).
	\end{equation*}
	We can consider $z_{jk}^R(Q), z_{jkl}^R(Q)$ to be rational functions in the variables $z_{jk}(Q)$,  $ z_{jkl}(Q)$ and $z(Q)$, see \cref{sec:addition_formulas}. Then it follows from the explicit formulas of these functions that 
	\begin{equation*}
	\deg z_{jk}^R \leq 3 \text{~  ,  ~} \deg z_{jkl}^R \leq 4 \text{~  and  ~} \deg z^R \leq 6.
	\end{equation*}
	Hence we obtain 
	\begin{align*}
	\deg \f_{R} \leq (\deg \f) \big(\max \{ \deg z_{jk}^{R}, \deg z_{jkl}^{R}, \deg z^R \}\big) = 6 \deg \f,
	\end{align*}
	and thus
	\begin{equation*}
	\deg \F \leq \deg \bigg( \sum_{i=0}^{L} c_i \f_{R} \bigg) \leq 6(L+1) (\deg \f).
	\end{equation*}
	
\end{proof}

\subsection{Bounds on the number of zeros of a system of polynomial equations over a finite field.}
%While \cref{prop:F_nonconst} gives an upper bound on the degree of $\F$, this does not immediately give an upper bound on the number of zeros of $\F$ over finite field $\Fqp$, since the zero set of $\F$ has dimension $1$. Now, we give an estimate for the number of zeros of $\F$ over finite field $\Fqp$.

Let $f_1,\dots, f_k\in \Fqp[X_1,\dots, X_m]$. We denote the vanishing set of $ f_1, \dots, f_k $ over $ \Fqp $ by 
$$
V_{\Fqp}(f_1,\dots, f_k)=\{\mathbf{x}\in\Fqp^m: f_1(\mathbf{x})=\dots=f_k(\mathbf{x})=0 \}, 
$$
and the vanishing set over the algebraic closure $ \Fqb $ by 
$$
V(f_1,\dots, f_k)=V_{\Fqb}(f_1,\dots, f_k).
$$
For each $ m \in \Z, m \geq 1 $, we define affine $ m $-space over $ \Fqb $ to be 
\begin{equation}
\mathbb{A}^m (\Fqb) =  \{ (x_1, \dots, x_m) : x_i  \in \Fqb, 1 \leq i \leq m \}. 
\end{equation}
The following result gives us bounds for the cardinality of algebraic sets over finite fields,
\cite[Corollary 2.2]{lachaud2015}. 
\begin{lemma}\label{lemma:bezout}
	Let  $f_1,\dots, f_k\in \Fqp[X_1,\dots, X_m]$ such that $  V(f_1,\dots, f_k) $ has dimension $ d $ in $ \mathbb{A}^m(\Fqb) $. Then
	$$
	|V_{\Fqp}(f_1,\dots, f_k)|=| V(f_1,\dots, f_k) \cap \Fqp^{m}|\leq (q^n)^d \prod_{i=1}^k \deg f_i.
	$$
\end{lemma} 

%The following lemma gives an estimate of the number of zeros of a rational function $ \F $ over a finite field $ \Fqp, n\geq 1 $, where $ \dim V(\F) = 1 $. It is a part of the proof of \cite[Theorem~4.1]{anupindiLinearComplexitySequences2021}, but for the sake of completeness, we state the result and include the proof.

\begin{lemma}\label{lem:Zeros_F_deg}
	Let $f_1,\dots,f_6$ be the defining equations of $U$ as in \eqref{eq:def_U}, let $\F \in \Fqp(U)$ be a non-constant rational function and let $G_1 / G_2$ be a representation of $\F \in \Fqp(\mathbf{\X})$ as a rational function. Then 
	\begin{equation}\label{eq:vanishing_F}
	|V_{\Fqp}(f_1,\dots,f_6,G_1)| \leq  216 q^n \deg \F. 
	\end{equation}
\end{lemma}
\begin{proof} Since $ U $ has dimension $ 2 $ and $ \F $ is non-constant on $ U $, $ V(f_1,\dots,\f_6, G_1) $ has dimension $ 1 $ in $ \mathbb{A}^8 $. Applying \cref{lemma:bezout}, we obtain
	\begin{equation*}
	|V_{\Fqp}(f_1,\dots,f_6,G_1)| \leq q^n \deg G_1 \prod_{i=1}^6 \deg f_i  \leq 216 q^n \deg \F.
	\end{equation*} 
\end{proof}

\subsection{Linear complexity}
We recall the following result on the linear complexity, see \cite[Lemma~ 6]{langeCertainExponentialSums2005}.
\begin{lemma}\label{lemma:linComp}
	Let $(s_n)$ be a linear recurrent sequence of order $L$ over $\Fq$ defined by a linear recursion
	$$
	s_{n+L}=c_0s_n +\dots + c_{L-1}s_{n+L-1}, \quad n\geq 0.
	$$
	Then for any $T\geq L+1$ and pairwise distinct positive integers $j_1,\dots, j_{T}$, there exist $\lambda_1,\dots, \lambda_T \in \Fq$, not all equal to zero, such that
	$$
	\sum_{i=1}^T\lambda_is_{n+j_i}=0, \quad n\geq 0.
	$$
\end{lemma} 

\subsection{Number of torsion elements}
We need the following result on the number of torsion elements in the Jacobian of hyperelliptic curves over finite fields. 
% Let $ J_C $ be the Jacobian of a hyperelliptic curve of genus $ g $. Let $ m \in \Z $. An element $ D $ is an $ m-$torsion element whenever $ mD = \cO $. 
\begin{lemma}\label{lem:torsion}
	%Let $ m $ be an integer such that characteristic of the field $ \Fq $ does not divide $ m $. 
	Let $ m $ be an integer coprime to the characteristic of $ \Fq $. Then, 
	\begin{equation*}
	\# \{D \in J_C(\Fqb) : mD = \cO \} = m^{2g}.
	\end{equation*}
\end{lemma}
For a proof, we refer to  \cite[Theorem~A.7.2.7]{hindryDiophantineGeometryIntroduction2000}.

\subsection{Collisions}

We now turn our attention to the collisions which can occur in \eqref{eq:Dm_def}. Let $T_k(Q)$ be the number of $k$-tuples $\mb = (m_0,\dots,m_{k-1}) \in \cR^k $ such that $D_{\mb} = Q$.
We recall the following result from \cite[Theorem 2]{ShparlinskiLange05}, which gives an upper bound for $ T_k(Q) $. This upper bound implies that the elements generated by \eqref{eq:Dm_def} do not take the same value too often and are sufficiently uniformly distributed.

\begin{prop}\label{lem:distinct_points}
	Let $C$ be a hyperelliptic curve of genus $2$ defined over $\Fq$ such that the characteristic polynomial of the Frobenius endomorphism $\chi_C$ is irreducible. Let $D \in J_C(\Fqn)$ of prime order $\ell$. Then for any integers $k$ and $e$ with $1\leq e \leq k$ and $q^{2e} \leq (q^{1/2} -1)^4 q^{-8} \ell $, and for every element $Q \in J_C(\Fqn)$, the bound $T_k(Q)\leq q^{2(k-e)}$ holds.
\end{prop}

The bound of \cref{lem:distinct_points} shows that if $ k $ is small and $q^{2k} \leq (q^{1/2} -1)^4 q^{-8} \ell $, then all the elements $D_{\mb}$ are distinct. We observe that if $ q^{2e} \ll \ell /q^8 $ then $q^{2e} \ll (q^{1/2} -1)^4 q^{-8}\ell $. For larger $k$, choosing $e$ maximal such that $ q^{2e} \ll \ell / q^8 $ yields 
\begin{equation}\label{eq:collision_bound}
T_k(Q) \leq \max \{1,  q^{2k - 2e} \} \ll \max \Big\{ 1, \frac{q^{2k + 8}}{\ell} \Big\} .
\end{equation}

\section{Proof of the main theorem}\label{sec:proof}
Let $ \chi_C(T) = T^4 + s_1T^3 + s_2 T^2 + s_1qT + q^2 $ be the characteristic polynomial of the Frobenius endomorphism for genus $ 2 $. The following result is a crucial step in the proof of the main theorem.  
\begin{lemma}\label{lem:sigma_distinct} 
	If $D \in J_C(\Fqn) $ has prime order $\ell$, and $\ell $ does not divide the constant term of $ \chi_C $, then $\sigma(D) \neq \cO$.
	%such that $\ell$ does not divide the constant term of $\chi_C$, i.e. $\ell \nmid q^2$, then $\sigma(D) \neq \cO$.
\end{lemma}
\begin{proof}
	If $ \sigma(D) = \cO $, then by definition of $ \chi_C $, we have that
	\begin{equation*}
	q^2D = -\sigma(D)^4 - s_1\sigma(D)^3 - s_2 \sigma(D)^2 - s_1q\sigma(D)  = \cO.
	\end{equation*}
	Thus, the order $ \ell $ of $ D $ divides $ q^2 $, which is the constant term of $ \chi_C $.
\end{proof} 
%Therefore, if $D_{\mb} \neq D_{\Bf{n}}$, then $\sigma^j(D_{\mb}) \neq \sigma^j(D_{\Bf{n}}), \new{\mb,\Bf{n} \in \cR^k, \mb \neq \Bf{n} }, j \in \Z, j \geq 1$.

\begin{proof}[Proof (\cref{thm:main_result})]
	We fix $ r = \max \{  \left\lfloor \frac{k}{4} \right\rfloor , 1 \} $.
	Let $\mb \in \cR^k$, we can write 
	\begin{equation*}
	\mb = (\Bf{\mu},\Bf{\nu}), \Bf{\mu} \in \cR^{r}, \Bf{\nu}\in \cR^{k-r} .
	\end{equation*}
	Let $N_r$ and $N_{k-r}$ be the number of distinct elements $D_{\Bf{\nu}}, \Bf{\nu} \in \cR^r$ and $\Bf{\nu} \in \cR^{k-r}$ respectively. We can assume that $ \ell \gg q^{\frac{3}{4}n +8} $, since otherwise \eqref{eq:thm_main} holds trivially. Therefore, $ \max \{q^{2r}, q^{2k-2r} \} \ll \ell/q^8  $. Hence, by \eqref{eq:collision_bound}, we obtain
	%\begin{equation*}
	%N_r \geq \frac{\# \cR^{r}}{\max_Q T_r(Q)}  \gg \frac{q^{2r}}{\max \{1, q^{2r+8}/ \ell \} } = \min \Big\{q^{2r},\frac{\ell}{q^8} \Big\} 
	%\end{equation*}
	%and similarly
	\begin{equation}\label{eq:bound_Nkr}
	N_{k-r} \geq \frac{\# \cR^{k-r}}{\max_Q T_{k-r}(Q)}  \gg \frac{q^{2(k-r)}}{\max \{1, q^{2(k-r) + 8}/ \ell \} } = \min \Big\{q^{2(k-r)},\frac{\ell}{q^8} \Big\} .
	\end{equation}
	Let $L$ be the linear complexity of the sequence $(w_{\mb})_{\mb \in \cR^k}$ as defined in \eqref{eq:seq_def}. We can assume that 
	\begin{equation}\label{eq:bound_L_thm}
	L < \min\left\{N_r, \frac{|J_C(\Fqp)|-|\Theta(\Fqp)|}{|\Theta(\Fqp)|+16}\right\}, 
	\end{equation}
	since otherwise the theorem holds trivially.
	
	Since by \eqref{eq:bound_L_thm}, we assume that $ L < N_r $, there exist $ L+1 $  vectors  $\Bf{d_0},\dots,\Bf{d_L}\in \cR^r$ such that $D_{\Bf{d_0}}, \dots , D_{\Bf{d_L}}$ are distinct.
	%By \cref{lem:sigma_distinct}, we obtain that $\sigma^{k-r}(D_{\Bf{d_0}}),\dots,\sigma^{k-r}(D_{\Bf{d_L}}) $ are also distinct.
	
	We fix these vectors and for each $ j = 0, \dots,L $ define the sequence 
	\begin{equation*}
	a_{j}(\Bf{s}) = w_{(\Bf{d_j},\Bf{s})}, \quad \Bf{s} \in \cR^{k-r},
	\end{equation*} 
	%\begin{equation*}
	%\f_{j}(\Bf{s}) = f(D_{(\Bf{d_j}, \Bf{s})}), \quad \Bf{s} \in %\cR^{k-r},
	%\end{equation*}
	where again the elements $ a_{j}(\Bf{s}) $ are arranged in a sequence by using lexicographic ordering for vectors $ \Bf{s} $. The sequences $ (a_{j}(\Bf{s}) )_{\Bf{s} \in \cR^{k-r}} $ are parts of $ (w_{\mb})_{\mb \in \cR^k} $, that is, they are consecutive elements in  $ (w_{\mb})_{\mb \in \cR^k} $, as $ \Bf{s} $ runs through $ \cR^{k-r} $. By \cref{lemma:linComp}, these sequences are linearly dependent, i.e. there exist constants $ c_0, \dots . c_L \in \Fqp $, not all zero, such that 
	%Hence, they satisfy the same linear recurrence relation of order $ L $. Then these sequences are in a vector space of dimension at most $ L $ so they are linearly dependant, i.e. there exist constants $ c_0, \dots . c_L \in \Fqp $, not all zero, such that 
	\begin{equation}\label{eq:seq_LC}
	c_0  w_{(\Bf{d_0}, \Bf{s})}  + \dots + c_L  w_{(\Bf{d_L}, \Bf{s})}  = 0 ,  \quad \Bf{s} \in \cR^{k-r}.  
	\end{equation} 
	%By \eqref{eq:Dm_def}, we can rewrite this as 
	%\begin{equation*}
	%c_0  f(D_{\Bf{d_0}} + \sigma^{r}(D_{\Bf{s}}) ) + \dots + c_L  f(D_{\Bf{d_L}} + \sigma^{r}(D_{\Bf{s}}) )  = 0 ,  \quad \Bf{s} \in \cR^{k-r}. 
	%\end{equation*} 
	Note that for $ \mb = (\Bf{d_j}, \Bf{s}) $, $ D_{\mb} = D_{(\Bf{d_j}, \Bf{s})} = D_{\Bf{d_j}} + \sigma^r(D_{\Bf{s}}) $ by \eqref{eq:Dm_def}.
	
	We would like to avoid collision of elements $  D_{\Bf{d_j}} , j \in \{0,\dots, L\}$ with $  \Theta(\Fqp) $. We claim that there exists an element $ R \in J_C(\Fqn)$ such that
	\begin{align}
	D_{\Bf{d_i}}  + R \notin \Theta(\Fqp), \quad &\text{~for~} 0 \leq i \leq L, \label{eq:R_cond_1}\\
	D_{\Bf{d_L}}  + R \neq -( D_{\Bf{d_j}}  + R),\quad &\text{~for~} 0 \leq j \leq L-1. \label{eq:R_cond_2}
	\end{align} 
	We count the number of elements $ R \in J_C(\Fqn) $ such that $ R $ does not satisfy \eqref{eq:R_cond_1} or \eqref{eq:R_cond_2}. There are at most $ (L+1)|\Theta(\Fqp)| $ choices for $ R $ such that $ 		 D_{\Bf{d_i}}  + R \in \Theta(\Fqp) $ for some $ 0 \leq i \leq L $.
	%Indeed, if \eqref{eq:R_cond_1} was not satisfied, then we have at most $ (L+1)|\Theta(\Fqp)| $ choices for $ R $ such that $ 	\sigma^{k-r}(D_{\Bf{d_i}}) + R \in \Theta(\Fqp) $ for $ 0 \leq i \leq L $. 
	
	Furthermore, if \eqref{eq:R_cond_2} was not satisfied, then we obtain that 
	\begin{equation*}
	-( D_{\Bf{d_L}}  +  D_{\Bf{d_j}} ) = 2R, \text{~for some~} 0 \leq j \leq L-1. 
	\end{equation*}
	By \cref{lem:torsion}, we obtain that there are at most $ 16 $ elements $ R \in J_C(\Fqb) $ such that $ 2R = \cO $.
	%of $ n $ torsion elements of $ J_C(\Fqb) $ is $ n^{2g} $, i.e for $ g=2=n $ there are at most $ 16 $ elements $ R \in J_C(\Fqb) $ such that $ 2R = \cO $. See  \cite[Theorem A.7.2.7]{hindryDiophantineGeometryIntroduction2000}. 
	Therefore, there are at most $ 16 L $ choices for $ R $ such that $2R = -( D_{\Bf{d_L}}  +  D_{\Bf{d_j}} ) $ for some $ j \in \{0,\dots,L-1\} $.
	By \eqref{eq:bound_L_thm}, we know that $$ |J_C(\Fqp)| - (L+1)|\Theta(\Fqp)| - 16L > 0, $$ hence there exists $ R \in J_C(\Fqp) $ such that \eqref{eq:R_cond_1} and \eqref{eq:R_cond_2} are satisfied.
	
	Let $R_i =  D_{\Bf{d_i}}  + R $. 
	Consider the function
	\begin{equation}\label{eq:F_thm}
	\F(Q) = \sum_{i=0}^{L} c_i \f(Q + R_i).
	\end{equation}
	By \cref{prop:F_nonconst} we know that $\F$ is non-constant and has degree at most $  6(L+1) (\deg \f) $.
	
	%We observe that if $ f(D_{\mb}) $ is defined for $ \mb = (\Bf{d_j},\Bf{s}), \Bf{d_j} \in \cR^r,\Bf{s} \in \cR^{k-r} $, then by \eqref{eq:Dm_def}, we can write $ f(D_{\mb}) = f(D_{\Bf{d_j}} + \sigma^{r}(D_{\Bf{s}})) $.
	%Writing $ D_{\mb} = D_{\Bf{d_j}} + \sigma^{r}(D_{\Bf{s}}) $, 
	
	We observe that if $ \F $ has a pole at $ Q $, then $ Q $ must have a form $ Q = \sigma^{r}(D_{\Bf{s}})  - R, \Bf{s} \in \cR^{k-r}, $ with $ Q \in \Theta(\Fqp) $ or $ Q \in \Theta(\Fqp) \pm R_i $ for $ 0 \leq i \leq L $. Hence, defining set $ \cS $ as follows ensures that for $ Q \in \cS $, the sum in \eqref{eq:F_thm} does not contain any poles.
	Define 
	\begin{align*}
	\cS = \{ Q \in J_C(\Fqn) :~ &Q = \sigma^{r}(D_{\Bf{s}})  - R, \Bf{s} \in \cR^{k-r}, \text{ with } Q \notin \Theta(\Fqp) \\* &\text{ and }  Q \pm R_i \notin \Theta(\Fqp), \text{ for } 0 \leq i \leq L \}.
	\end{align*} 
	%The definition of set $ \cS $ ensures that for $ Q \in \cS $, the sum in \eqref{eq:F_thm} does not contain any poles. 
	Hence by \eqref{eq:seq_def} and \eqref{eq:seq_LC}, $ \F(Q) = 0 $ for $ Q \in \cS $. 
	
	Now we give a lower bound for $ |\cS| $. We observe that if $ D $ has prime order $ \ell $, then $ \sigma^j(D)$ also has order $ \ell $, since $ \sigma $ is an endomorphism and hence additive. Combining this with \eqref{eq:Dm_def}, we see that if $ \ell D_{\mb} = \cO $, then either $D_{\mb}= \cO$ or it has order $ \ell $, since $ \ell $ is prime. Hence, by \cref{lem:sigma_distinct}, we obtain that if $D_{\mb} \neq D_{\Bf{n}}$, then $\sigma^j(D_{\mb}) \neq \sigma^j(D_{\Bf{n}}), \mb,\Bf{n} \in \cR^k, j \in \Z, j \geq 1$. Therefore, the number of distinct elements $ \sigma^{r}(D_{\Bf{s}}) - R, \Bf{s} \in \cR^{k-r} $ is $ N_{k-r} $.
	We observe that for $ Q = \sigma^{r}(D_{\Bf{s}})  - R, \Bf{s} \in \cR^{k-r} $,
	\begin{equation}\label{eq:bad_Q_1}
	|\{Q \in J_C(\Fqn) :  Q \in \Theta(\Fqn) \}| \leq |\Theta(\Fqp)|
	\end{equation}
	and
	\begin{equation}\label{eq:bad_Q_2}
	|\{Q \in J_C(\Fqn) : Q \pm R_i \in \Theta(\Fqn) \}| \leq  2 (L+1)|\Theta(\Fqp)|.
	\end{equation}
	%Hence, counting the number of distinct elements $ \sigma^{r}(D_{\Bf{s}})  $ is the same as counting the number of distinct elements $  D_{\Bf{s}} $.There are at most $|\Theta(\Fqp)|$ elements $Q $ with $ Q \in \Theta(\Fqp)$ and similarly for each $i,~ 0\leq i \leq L$, there are $2|\Theta(\Fqp)|$ elements  $Q$ with $ Q \pm R_i \in \Theta(\Fqp)$. 
	Hence, by \eqref{eq:bad_Q_1} and \eqref{eq:bad_Q_2} we obtain, 
	\begin{equation}\label{eq:S_lowerbound}
	|\cS| \geq N_{k-r} - 2 (L+1)|\Theta(\Fqp)| - |\Theta(\Fqp)|.
	\end{equation}
	
	To give an upper bound for $ |\cS| $, 
	%we first observe that by \eqref{eq:bound_L_thm}, \eqref{eq:claim} is satisfied since
	%\begin{equation*}
	%|\cS| \geq N_{k-r} - 2 (L+1)|\Theta(\Fqp)| - |\Theta(\Fqp)| > 2|\Theta(\Fqp)| + 2(q^n +1).
	%\end{equation*}  
	we use \cref{lem:Zeros_F_deg} and \eqref{eq:deg_F} to obtain
	\begin{equation}\label{eq:S_upperbound}
	|\cS| \leq 216q^n \deg \F \leq 1296(L+1)q^n \deg \f.
	\end{equation}
	Combining equations \eqref{eq:S_lowerbound} and \eqref{eq:S_upperbound} gives us
	\begin{equation}\label{eq:estimate_L}
	L \geq \frac{N_{k-r} - 3|\Theta(\Fqp)| - 1296 q^n \deg \f}{1296 q^n \deg \f + 2|\Theta(\Fqp)|}.
	\end{equation}
	By \eqref{eq:number_of_points}, we can estimate the size of $ |\Theta(\Fqn)| $.
	Substituting the lower bound on $ N_{k-r} $ as given in \eqref{eq:bound_Nkr} into \eqref{eq:estimate_L}, we obtain 
	\begin{equation*}
	L(w_{\mb}) \gg   \frac{\min\{ q^{3k/2},\ell/q^8 \}}{ q^n \deg \f } .
	\end{equation*} 
	
\end{proof}	

\section*{Acknowledgements}
	The author was supported by the Austrian Science Fund FWF Project P31762. I am very grateful to L{\'a}szl{\'o} M{\'e}rai for valuable discussions and comments.

\appendix
\section{}
\begingroup
\allowdisplaybreaks
Let $C$ be the hyperelliptic curve defined by \eqref{def:hec} with
$$
f(X) = X^5 + b_1X^4 + b_2X^3 + b_3X^2 + b_4X + b_5 \in \Fq[X],
$$
for the finite field $\Fq$ with characteristic $p \geq 3 $.
\subsection{Defining equations of the Jacobian}
\label{sec:def_eq_J}
Let  
$$
S= \Fq[\X_0, \X_{11},\X_{12},\X_{22},\X_{111},\X_{112},\X_{122},\X_{222},\X] $$ 
be a polynomial ring over field $\Fq$, with characteristic $p \geq 3 $. Following \cite{grant1990formal}, in particular Theorem 2.5, Theorem 2.11 and Corollary 2.15, we define $f_i $ as follows: 
\begin{align*}
f_0 = &\X^2 + \X_{11}^2\X_{12} + b_1\X_{11}^2\X_{22} +   b_2\X_{11}^2\X_{12}\X_{22} - b_3\X_{11}\X_{22}^2 + b_4\X_{12}\X_{22}^2 \\
&-b_5\X_{22}^3 + 2b_1\X\X_{11} -2b_2\X\X_{12} + 2b_3\X\X_{22}+ (b_3-b_1b_2)\X_{11}\X_{12} \\
&+(b_2^2-b_1b_3)\X_{11}\X_{22} + (b_1b_4-b_2b_3-b_5)\X_{12}\X_{22}-b_1b_5\X_{22}^2 \\
& +2(b_1b_3 -b_2^2)\X + (b_1b_4-b_5)\X_{11} +b_2(b_2^2 - b_1b_3 )\X_{12} \\
&(b_3b_4 - b_2b_5)\X_{22} + b_1b_3b_4 - b_2^2b_4 - b_3b_5, \\
f_1 =&2\X - \X_{11}\X_{22} + \X_{12}^2 - b_2\X_{12} + b_4, \\
f_2 =&\X_{112} -\X_{222}\X_{12} + \X_{122}\X_{22}, \\
f_3 =&\X_{111} + \X_{222}\X_{11} + \X_{122}\X_{12} - 2\X_{112}\X_{22} - 2b_1\X_{112} + b_2\X_{122}, \\
f_4 =&\X_{122}^2 - \X_{11}\X_{22}^2 + 2\X\X_{22} + \X_{11}\X_{12} - b_1\X_{11}\X_{22} - b_2\X_{12}\X_{22} \\
& +2b_1\X - b_1b_2\X_{12} + b_4\X_{22} + b_1b_4 - b_5 , \\
f_5 =& \X_{222}^2 - \X_{22}^3 - \X_{12}\X_{22} - b_1\X_{22}^2 - \X_{11} - b_2\X_{22} - b_3 , \\
f_6 =&\X_{122}\X_{222} - \X_{12}\X_{22}^2 + \X -b_2\X_{12} - b_1\X_{12}\X_{22}, \\
f_7 =&\X_{111}^2 - \X_{11}^3 - b_3\X_{11}^2 -b_4\X_{11}\X_{12} + 3b_5\X_{11}\X_{22} + 2b_5\X \\
&+(4b_1b_5 - b_2b_4)\X_{11} - 3b_2b_5\X_{12} + (4b_3b_5-b_4^2)\X_{22} \\
&4b_1b_3b_5 + b_4b_5 - b_1b_4^2 - b_2^2b_5, \\
f_8 =& -\X_{111}\X_{112} + b_1\X_{111}\X_{122} -b_2\X_{112}\X_{122} + b_3\X_{112}\X_{222} \\
&-b_4\X_{122}\X_{222} + b_5\X_{222}^2 -\X^2 -b_1 \X \X_{11} + b_2 \X \X_{12} - b_3 \X \X_{22} \\
&-b_3\X_{11}\X_{12} + b_1b_3\X_{11}\X_{22} - (b_5 + b_1b_4)\X_{12}\X_{22} + 2b_1b_5\X_{22}^2 \\
& -2(b_1b_3 + b_4)\X  + (2b_2b_4 + b_1b_2b_3 +b_1b_5 - b_3^2 -b_1^2b_4)\X_{12} \\
&- 2b_5\X_{11} + 2b_5(b_1^2 - b_2)\X_{22} + b_1b_2b_5 - b_1b_3b_4 - 2b_3b_5 , \\
f_9 =& \X_{122}^2 - \X_{111}\X_{122} + \X_{11}\X - b_3\X_{11}\X_{22} + 2b_4\X_{12}\X_{22} -3b_5\X_{22}^2 \\
&+2b_3\X + (b_1b_4 - b_2b_3 - b_5)\X_{12} - 2b_1 b_5 \X_{22} + b_3b_4 - b_2b_5, \\
f_{10} =&\X_{111}\X_{222} - \X_{112}\X_{122} -2\X\X_{12} + \X_{11}^2 - 2b_1 \X_{11}\X_{12} \\
&+3b_2 \X_{11}\X_{22} - 2b_3 \X_{12}\X_{22} + b_4\X_{22}^2 -5b_2\X +b_3\X_{11} \\
&+ (3b_2^2 - 2b_1b_3)\X_{12} + (b_1b_4 - b_5)\X_{22} - 2b_2b_4, \\
f_{11}=&\X_{122}^2 - \X_{112}\X_{222} + \X_{22}\X + 2\X_{11}\X_{12} - b_1\X_{11}\X_{22} + 2b_1\X \\
&+(b_3-b_1b_2)\X_{12} + b_1b_4 - b_5, \\
f_{12} =&\X_{111}\X_{12} -\X_{112}\X_{11} - b_4\X_{122} + 2b_5\X_{222}, \\
f_{13} =& 2\X_{122}\X_{11} - \X_{112}\X_{12} -\X_{111}\X_{22} - b_2\X_{112} + 2b_3\X_{122} - b_4\X_{222}.
\end{align*}
One can show that $f_0\in \langle f_4,f_5,f_6\rangle$
and the vanishing locus of these polynomials homogenized with respect to the variable $\X_0$ forms a set of defining equations for the Jacobian $J_C$, i.e 
$$
J_C = V(f_1^h,\dots,f_{13}^h) = \{ z \in \mathbb{P}^8(\Fqb) : f_i^h(z) = 0, 1 \leq i \leq 13 \}
$$

\subsection{Addition formulas}
\label{sec:addition_formulas}
For $D = (x_1,y_1) + (x_2,y_2) - 2\cO \in J_C(\Fq)\setminus \Theta(\Fq)$, \eqref{eq:iota} gives us $ \iota(D) $, where the components $z_{jk}, z_{jkl}$ of $\iota(D)$ can be expressed as rational functions in the coordinates $(x_1,y_1)$ and $(x_2,y_2)$. For the sake of completeness, we collect the addition formulas as given in \cite[Theorem 3.3]{grant1990formal} and as explicitly computed in \cite[Appendix A.3]{anupindiLinearComplexitySequences2021}.

\begin{align*}
z_{ij}(Q+R) =& -z_{ij}(Q)-z_{ij}(R) + \frac{1}{4}\left( \frac{q_{i}(Q,R)}{q(Q,R)} \right) \left( \frac{q_{j}(Q,R)}{q(Q,R)} \right) \\
&-\frac{1}{4} \left( \frac{q_{ij}(Q,R)}{q(Q,R)} \right)\\
z_{111}(Q+R)= &-\frac{1}{2}z_{111}(Q)-\frac{1}{2}z_{111}(R) + \frac{3}{16}\frac{q_{1}(Q,R)q_{11}(Q,R)}{q(Q,R)^2}-\frac{1}{16}\frac{q_{111}(Q,R)}{q(Q,R)} \\
&- \frac{1}{8} \left( \frac{q_{1}(Q,R)}{q(Q,R)} \right) ^3 
+ \frac{3}{4}(z_{11}(Q)+z_{11}(R))\frac{q_{1}(Q,R)}{q(Q,R)}, \\
z_{112}(Q+R)= &-\frac{1}{2}z_{112}(Q)-\frac{1}{2}z_{112}(R) + \frac{1}{16}\frac{q_{2}(Q,R)q_{11}(Q,R)}{q(Q,R)^2}\\ &+\frac{1}{8}\frac{q_{1}(Q,R)q_{12}(Q,R)}{q(Q,R)^2} -\frac{1}{16}\frac{q_{112}(Q,R)}{q(Q,R)}
- \frac{1}{8}\frac{q_{2}(Q,R)(q_{1}(Q,R))^2}{q(Q,R)^3} \\  
&+ \frac{3}{8}(z_{11}(Q)+z_{11}(R))\frac{q_{2}(Q,R)}{q(Q,R)} + \frac{3}{8}(z_{12}(Q)+z_{12}(R))\frac{q_{1}(Q,R)}{q(Q,R)} \\
z_{122}(Q+R)= &-\frac{1}{2}z_{122}(Q)-\frac{1}{2}z_{122}(R) + \frac{1}{16}\frac{q_{1}(Q,R)q_{22}(Q,R)}{q(Q,R)^2}\\ &+\frac{1}{8}\frac{q_{2}(Q,R)q_{12}(Q,R)}{q(Q,R)^2} -\frac{1}{16}\frac{q_{122}(Q,R)}{q(Q,R)} \\
&- \frac{1}{8}\frac{q_{1}(Q,R)(q_{2}(Q,R))^2}{q(Q,R)^3}  
+ \frac{3}{4}(z_{12}(Q)+z_{12}(R))\frac{q_{2}(Q,R)}{q(Q,R)} \\
z_{222}(Q+R)= &-\frac{1}{2}z_{222}(Q)-\frac{1}{2}z_{222}(R) + \frac{3}{16}\frac{q_{2}(Q,R)q_{22}(Q,R)}{q(Q,R)^2}-\frac{1}{16}\frac{q_{222}(Q,R)}{q(Q,R)} \\
&- \frac{1}{8} \left( \frac{q_{2}(Q,R)}{q(Q,R)} \right) ^3 
+ \frac{3}{4}(z_{22}(Q)+z_{22}(R))\frac{q_{2}(Q,R)}{q(Q,R)} \\
z(Q+R) =& \frac{1}{2}(z_{11}(Q+R)z_{22}(Q+R) - z_{11}^2(Q+R) + b_2z_{12}(Q+R) - b_4 )
\end{align*}

To evaluate the addition formulas above, we need the following rational functions:
\newcommand{\p}[2]{z_{#1#2}}
\newcommand{\pp}[3]{2z_{#1#2#3}}
\newcommand{\po}{(2z - b_2 z_{12} + b_4)}

\begin{align*}
q(Q,R)= &z_{11}(Q)-z_{11}(R)+z_{12}(Q)z_{22}(R)-z_{12}(R)z_{22}(Q), \\
q_1(Q,R) = &\pp111(Q) - \pp111(R) + \pp112(Q)\p22(R) - \pp112(R)\p22(Q) \\
& +\pp122(R)\p12(Q) - \pp122(Q)\p12(R), \\ 
q_2(Q,R) =& \pp112(Q) - \pp112(R) + \pp122(Q)\p22(R) - \pp122(R)\p22(Q) \\
& +\pp222(R)\p12(Q) - \pp222(Q)\p12(R), \\
q_{11}(Q,R)=& 4b_3 q(Q,R) + 4b_4(\p12(Q)-\p12(R)) + 4(\po(Q)\p12(R))\\
&-4(\po(R)\p12(Q))-8b_5(\p22(Q) - \p22(R)) \\ &+2(\pp112(Q)\pp122(R) - \pp112(R)\pp122(Q)),\\
q_{12}(Q,R) =& 4b_3(\p12(Q)-\p12(R)) + 2b_2(\p12(Q)\p22(R)) \\ &-2b_2(\p12(R)\p22(Q)) - 4(\p11(Q)\p12(R)-\p11(R)\p12(Q))\\
& +2(\po(Q)\p22(R)-\po(R)\p22(Q)) \\
& - 2b_4(\p22(Q) - \p22(R)) + \pp222(R)\pp112(Q)-\pp222(Q)\pp112(R), \\
q_{22}(Q,R) =& 8b_1(\p12(Q)\p22(R)-\p12(R)\p22(Q)) + 4b_2\p12(Q)\\
& -4b_2\p12(R) -8(\p11(Q)\p22(R)-\p11(R)\p22(Q)) \\
& - 4(\po(Q)-\po(R)) \\
& + 2(\pp122(Q)\pp222(R) -\pp122(R)\pp222(Q)),\\
q_{111}(Q,R) =& 4b_3q_1(Q,R) \\
& + 4(\pp111(Q)\p22(Q)\p12(R)-\pp111(R)\p22(R)\p12(Q)) \\
& + \pp122(R)(2\p12(Q)(6\p11(Q) - 2\p11(R)+4b_3)-4b_4\p22(Q))\\
& - \pp122(Q)(2\p12(R)(6\p11(R) - 2\p11(Q)+4b_3)-4b_4\p22(R))\\
& + \pp112(Q)(\p12(R)(12\p12(R) - 8\p12(Q)+4b_2)+4b_4)\\
& - \pp112(R)(\p12(Q)(12\p12(Q) - 8\p12(R)+4b_2)+4b_4)\\
q_{112}(Q,R) = &~\pp222(Q)
\left(4 \p11(Q) \p12(R) - 4 \p12(R) b_{3} - 8 b_{5}
\right) \\ 
&+\pp112(Q)
\left(- 4 \p11(R) + 4 \p12(R) \p22(Q) + \p12(R) \left(12 \p22(R) + 8 b_{1}\right)
\right) \\ 
&+\pp112(R)
\left(4 \p11(Q) + \p12(Q) \left(- 12 \p22(Q) - 4 \p22(R) - 8 b_{1}\right) - 4 b_{3}
\right) \\ 
&+\pp122(Q)
(- 8 \p11(R) \p22(R) - 8 \p12(Q) \p12(R) - 4 \p12(R)^{2}  \\
&+ 4 \p22(R) b_{3} + 4 b_{4} - 4 \p12(R) b_{2} ) \\ 
&+\pp122(R)
\left(8 \p11(Q) \p22(Q) + 4 \p12(Q)^{2} + \p12(Q) \left(8 \p12(R) + 4 b_{2}\right) \right.\\
&- \left. 4 \p22(Q) b_{3} - 4 b_{4} \right) + \pp112(Q)4 b_{3}\\ 
&+\pp222(R)
\left(\p12(Q) \left(- 4 \p11(R) + 4 b_{3}\right) + 8 b_{5}
\right) \\
q_{122}(Q,R) = &~\pp112(R)
\left(- 6 \p22(Q)^{2} + \p22(Q) \left(- 2 \p22(R) - 4 b_{1}\right) - 2 b_{2}
\right) \\ 
&+\pp122(R)
\left(- 4 \p11(Q) + \p22(Q) \left(4 \p12(R) - 2 b_{2}\right) - 4 b_{3}
\right) \\ 
&+\pp222(Q)
\left(2 \p11(Q) \p22(R) - 4 \p11(R) \p22(R) - 2 \p12(R)^{2}   
\right) \\ 
&+\pp112(Q)
\left(2 \p22(Q) \p22(R) + 6 \p22(R)^{2} + 4 \p22(R) b_{1} + 2 b_{2}
\right) \\ 
&+\pp222(R)
\left(4 \p11(Q) \p22(Q) - 2 \p11(R) \p22(Q) + 2 \p12(Q)^{2} 
\right) \\ 
&+\pp122(Q)
\left(4 \p11(R) - 4 \p12(Q) \p22(R) + 2 \p22(R) b_{2} + 4 b_{3}
\right) \\
&-\pp222(Q) (2 b_{4} + 4 \p12(R) b_{2}) \\
&+ \pp222(R) (2 b_{4} + 4 \p12(Q) b_{2}) \\
q_{222}(Q,R) = &~\pp222(R)
\left(- 12 \p11(Q) + 4 \p11(R) + \p12(Q) \left(12 \p22(Q) + 16 b_{1}\right)
\right) \\ 
&+\pp122(R)
\left(- 8 \p12(Q) - 8 \p12(R) - 12 \p22(Q)^{2} - 16 \p22(Q) b_{1} - 8 b_{2}
\right) \\ 
&+\pp112(Q)
\left(- 4 \p22(Q) - 8 \p22(R)
\right) \\ 
&+\pp222(Q)
\left(- 4 \p11(Q) + 12 \p11(R) + \p12(R) \left(- 12 \p22(R) - 16 b_{1}\right)
\right) \\ 
&+\pp112(R)
\left(8 \p22(Q) + 4 \p22(R)
\right) \\ 
&+\pp122(Q)
\left(8 \p12(Q) + 8 \p12(R) + 12 \p22(R)^{2} + 16 \p22(R) b_{1} + 8 b_{2}
\right)
\end{align*}
\endgroup

\bibliographystyle{amsplain}
\bibliography{Bibliography}
\end{document}